\documentclass{amsart}
\usepackage{amsfonts}
\usepackage{amssymb}
\usepackage{amsmath}
\usepackage{amsthm}
\usepackage{moreverb}
\usepackage{fancybox}
\usepackage{fancyvrb}
\usepackage{comment}
\usepackage{color}
\usepackage{multirow}
\usepackage[all]{xy}

\theoremstyle{plain}
\newtheorem{theorem}{Theorem}[section]
\newtheorem{lemma}[theorem]{Lemma}

\newtheorem{cor}[theorem]{Corollary}
\newtheorem{prop}[theorem]{Proposition}

\newtheorem{question}{Question}

\newtheorem{fact}{Fact}
\newtheorem{obs}[theorem]{Observation}

\theoremstyle{definition}
\newtheorem{definition}[theorem]{Definition}

\theoremstyle{definition}
\newtheorem*{remark}{Remark}

\newtheorem*{ack}{Acknowledgement}

\newcommand{\upto}{\upharpoonright}
\newcommand{\fr}{\mbox{}^\smallfrown}

\newcommand{\om}{\omega}
\newcommand{\pcolon}{\colon\!\!\!\subseteq}

\newcommand{\subT}{\leq_{\rm subT}}
\newcommand{\eqsubT}{\equiv_{\rm subT}}

\newcommand{\Eff}{\mathcal{E}\!f\!f}

\newcommand{\embed}{\hookrightarrow}

\title{The subTuring degrees}
\author{}
\date{}
\author{Takayuki Kihara and Keng Meng Ng}

\begin{document}
\maketitle

\begin{abstract}
In this article, we introduce a notion of reducibility for partial functions on the natural numbers, which we call subTuring reducibility.
One important aspect is that the subTuring degrees correspond to the structure of the realizability subtoposes of the effective topos.
We show that the subTuring degrees (that is, the realizability subtoposes of the effective topos) form a dense non-modular (thus, non-distributive) lattice.
We also show that there is a nonzero join-irreducible subTuring degree (which implies that there is a realizability subtopos of the effective topos that cannot be decomposed into two smaller realizability subtoposes).
\end{abstract}

\section{Introduction}

\subsection{Background}
\label{sec:background}

The analysis of the degrees of non-computability is a central subject in computability theory \cite{Cooper,OdiBook}.
The key notion used to compare degrees of non-computability of {\em total} functions on the natural numbers is {\em Turing reducibility}, and a huge amount of research has been done on it.
The best known reducibility for {\em partial} functions on the natural numbes is Kleene's relative partial computability or nondeterministic Turing reducibility \cite[Chapter 11]{Cooper}, which coincides with {\em enumeration reducibility} for the graphs of partial functions.
Therefore, it has been widely believed that the study of degrees of partial functions can be absorbed into the theory of enumeration degrees --- another field that has developed greatly, like the theory of Turing degrees.

However, unlike the case of total functions, other candidates for the notion of reducibility for partial functions have been proposed.
Sasso \cite{Sas75} used a type of relative computation with sequential (non-parallel) access to an oracle to introduce a reducibility notion between partial functions.
This degree structure, which Sasso called the $T$-degrees of partial functions, has been little studied since then in degree theory, but the importance of Sasso's degree structure has been highlighted in a context quite different from that of degree theory.

In constructive mathematics, the pioneering work was done by Goodman \cite{Goo78}, who used a model of relative computation similar to Sasso's to show the so-called Goodman's theorem:
The system ${\sf HA}^\om$ of finite type Heyting arithmetic plus the axiom of choice ${\sf AC}$ is conservative over Heyting arithmetic ${\sf HA}$.
The proof uses the realizability interpretation relative to a generic partial choice function; see also \cite{Bee79,BeSl18}.
The notion of computability relative to a partial function used here is precisely the notion of {\em subTuring reducibility} (Definition \ref{def:subTuring}), which is a slight modification of Sasso's definition (first formally introduced by Madore \cite{Mad12} and later by Kihara \cite{Kih22}).

The exact same reducibility has also been used in van Oosten's semantical proof of de Jongh's theorem \cite{vO91}, for example:
For a formula $A$ in intuitionistic predicate calculus ${\sf IQC}$, ${\sf IQC}$ proves $A$ iff ${\sf HA}$ proves every arithmetical substitution instance of $A$.
The proof of backward direction is done by replacing predicate symbols with mutually generic arithmetical formulas.
The tool used here is a pca-valued sheaf (pca stands for partial combinatory algebra).
Each pca used in the proof is of the form $K_1[f]$, where $K_1[f]$ represents a model of computability relative to a partial function $f$.
To be more precise, $K_1$ stands for Kleene's first algebra; that is, the standard model of computability on the natural numbers, and the relative computability is precisely subTuring reducibility.

Later, van Oosten \cite{vO99} used the same relative computability, which he called a dialogue, to introduce the type structure of finite type functionals based on partial functions.
Given a partial combinatory algebra $A$, the dialogue for a partial function $f$ yields a new algebra $A[f]$; see e.g.~\cite{vO06}.
As one of the notable results, Faber-van Oosten \cite[Corollary 2.11]{FavO14} showed that a realizability subtopos (that is, a realizability topos which is a subtopos) of the effective topos $\Eff$ is nothing but the realizability topos over $K_1[f]$ for some partial function $f$.
This shows that the structure of the subTuring degrees of partial functions on the natural numbers corresponds exactly to the structure of the realizability subtoposes of the effective topos $\Eff$.

In short, a modification of Sasso's model of sequential computation, i.e., subTuring reducibility, was essential in various works on constructive mathematics, realizability topos theory, etc.
For these reasons, Kihara \cite{Kih22} has isolated this notion of subTuring reducibility, and raised the analysis of its structure as an important problem.
The issue of structural analysis of subTuring degrees has also been raised by Madore \cite{Mad12}.

This article addresses the problem \cite{Mad12,Kih22} of analyzing the structure of the subTuring degrees (in other words, the structure of the realizability subtoposes of $\Eff$).
In this article, we only aim to perform a pure degree-theoretic analysis.
As a pure degree theory, our research complements the perspective on partial functions that classical computability theorists have overlooked for many years.
Our work provides novel techniques for computability theory and is interesting in its own right as a pure degree theory.
As a constructive mathematics, in the subsequent article \cite{KiNg}, we will apply the powerful techniques we have developed in this article to various model constructions in constructive mathematics.

\subsection{Summary of results}

In Section \ref{sec:introduction}, we introduce the notion of subTuring reducibility (Definition \ref{def:subTuring}).
In Section \ref{sec:basic}, we show that the subTuring degrees form a lattice (Proposition \ref{prop:subTuring-lattice}), but not $\sigma$-complete.
Indeed, no strictly increasing sequence of subTuring degrees has a supremum (Proposition \ref{prop:no-strictly-increasing}).
Similarly, no strictly decreasing sequence of subTuring degrees has an infimum (Proposition \ref{prop:no-strictly-decreasing}).
We introduce the notion of subTuring jump, and we show that the subTuring jump has a fixed point (Theorem \ref{thm:jump-fixed-point}).
In Section \ref{sec:theorem}, we show that the subTuring degrees are dense (Theorem \ref{thm:subTuring-dense}); indeed, any nonempty interval of subTuring degrees has continuum many incomparable degrees (Theorem \ref{thm:dense-continuum}).
We introduce the notion of quasiminimal degree, and we prove the quasiminimal jump inversion theorem for subTuring degrees; that is, any subTuring degree above the halting problem is the subTuring jump of a quasiminimal subTuring degree (Theorem \ref{thm:quasiminimal-jump-inversiona}).
We show that the subTuring meet of partial functions cannot be the degree of a total function (Theorem \ref{thm:no-minimal-pair}).
In particular, no minimal pair of subTuring degrees of total functions exists (Corollary \ref{cor:no-minimal-pair}).
We show that the lattice of the subTuring degrees is not distributive (Theorem \ref{thm:subTuring-nondistributive}); indeed, not modular (Theorem \ref{thm:subTuring-nonmodular}).
Hence, the realizability subtoposes of the effective topos form a dense, non-distributive, non-modular lattice.
We finally show that there is a nonzero join-irreducible subTuring degree (Theorem \ref{thm:subTuring-join-irreducible}), which implies that there is a realizability subtopos of the effective topos that cannot be decomposed into two smaller realizability subtoposes (with respect to the ordering by geometric inclusions).

\section{Preliminaries}\label{sec:introduction}

\subsection{Notations}
For the basics of computability theory, see \cite{Cooper,OdiBook}.
Let $\om$ denote the set of all natural numbers, and $\om^{<\om}$ denote the set of all finite strings on $\om$.
We use the notation $f\pcolon X\to Y$ to indicate that $f$ is a partial function from $X$ to $Y$; that is, the domain of $f$ is a subset of $X$ and the codomain is $Y$.
Here, the domain, ${\rm dom}(f)$, is the set of all $x\in X$ such that $f(x)$ is defined.
If $x\in{\rm dom}(f)$, we often write $f(x)\downarrow$; otherwise $f(x)\uparrow$.
For a set $A\subseteq X$, let $f\upto A$ denote the restriction of $f$ up to $A$; that is, ${\rm dom}(f\upto A)={\rm dom}(f)\cap A$, and $(f\upto A)(x)=f(x)$ for any $x\in{\rm dom}(f\upto A)$.

\subsection{Definition}

In Sasso's model of relative computation, access to an oracle is always sequential, so parallel access is not allowed.
In other words, if we make a query to an oracle during a computation, we must wait until the oracle replies.
However, if the oracle is a partial function, it is possible that the oracle will not return forever, in which case the computation will be stuck.
Therefore, a query we make must always be a value in the domain of the oracle if we want the computation to halt.

Let us formulate this Sasso's idea rigorously.
For a partial computable function $\Phi\pcolon\om^{<\om}\to 2\times\om$, the $g$-relative sequential computation $\Phi[g]$ is defined as follows:
\begin{enumerate}
\item Assume that $n$ is given as input, the oracle's responses $(a_0,\dots,a_{s-1})$ are given by the $s$th round, and the computation has not yet halted.
\item If $\Phi(n,a_0,\dots,a_{s-1})$ does not halt, then neither does $\Phi[g](n)$ (written $\Phi[g](n)\uparrow$).
Otherwise, we have $\Phi(n,a_0,\dots,a_{s-1})\downarrow=\langle i_s,q_s\rangle$.
\begin{enumerate}
\item If $i_s=1$, the computation halts and $q_s$ is the output of this computation.
\item If $i_s=0$, proceed to the next step (3).
\end{enumerate}
\item If $q_s\in{\rm dom}(g)$, the computation continues with $a_s=g(q_s)$.
If $q_s\not\in{\rm dom}(g)$, the computation never halts.
\end{enumerate}

If the computation arrives at (2a) above, we write $\Phi[g](n)\downarrow=q_s$.
If the computation goes to (2b), we call $q_s$ a query.
Note that this sequential computation has the following monotonicity:
$f\subseteq g$ implies $\Phi[f]\subseteq\Phi[g]$.

\begin{definition}\label{def:subTuring}
For partial functions $f,g\pcolon\om\to\om$, we say that {\em $f$ is subTuring reducible to $g$} (written $f\subT g$) if there exists a partial computable function $\Phi$ such that $f\subseteq\Phi[g]$; that is,  for any $n\in{\rm dom}(f)$ we have $\Phi[g](n)\downarrow=f(n)$.
\end{definition}

Roughly speaking, $f\subT g$ if there exists a Turing functional $\Phi$ which computes ${\Phi^g(n)\downarrow}=f(n)$ without making a query outside of ${\rm dom}(g)$, whenever an input $n$ is given from ${\rm dom}(f)$.
Hereafter, $\Phi[g]$ is sometimes written as $\Phi^g$.

\begin{remark}
Sasso's Turing reducibility $f\leq_Tg$ is defined as $f=\Phi[g]$; that is, it requires the condition ${\rm dom}(f)={\rm dom}(\Phi[g])$.
Thus, $f\subT g$ if and only if $f$ has an extension $\hat{f}$ with $\hat{f}\leq_Tg$.
Definition \ref{def:subTuring} is first proposed by Madore \cite{Mad12}.
\end{remark}

In particular, if we consider only total functions, it agrees exactly with the usual Turing reducibility.

\begin{obs}
If $g$ is total, then $g\leq_Tf$ if and only if $g\subT f$.
\end{obs}

\subsection{Basics}
Restricting to binary functions does not change the structure.
\begin{obs}
Any partial function $f\pcolon\om\to\om$ is subTuring equivalent to a partial binary function $G_f\pcolon \om\to 2$.
\end{obs}

\begin{proof}
Let $G_f$ be the graph of $f$ in the following sense:
\[
G_f(n,m)=
\begin{cases}
1 & \mbox{ if }{f(n)\downarrow}=m\\
0 & \mbox{ if }{f(n)\downarrow}\not=m\\
\uparrow & \mbox{ if }f(n)\uparrow
\end{cases}
\]

For $G_f\subT f$, to compute $G_f(n,m)$, just ask for the value of $f(n)$, where $(n,m)\in {\rm dom}(G_f)$ implies $n\in{\rm dom}(f)$.
For $f\subT G_f$, to compute $f(n)$, ask for the values of $G_f(n,0),G_f(n,1),G_f(n,2),\dots$, where $n\in {\rm dom}(f)$ implies $(n,m)\in{\rm dom}(G_f)$ for any $m$.
If $n\in{\rm dom}(f)$ then the answer $G(n,k)=1$ will be returned at some point, so we can output $k$ at that time.
\end{proof}

A subTuring degree is {\em total} if it contains a total function.

\begin{obs}\label{obs:ce-domain}
If the domain of $f\pcolon\om\to\om$ is c.e., then $f$ has a total subTuring degree.
Similarly, if the domain of $f\pcolon \om\to\om$ is computable relative to $f$, i.e., $\chi_{{\rm dom}(f)}\subT f$, then $f$ has a total subTuring degree.
\end{obs}

First, the minimum requirement for a reducibility notion is that it be preordered.

\begin{obs}
$\subT$ is transitive.
\end{obs}

\begin{proof}
Assume that $f\subT g\subT h$.
Then we get $f\subseteq\Phi[g]$ and $g\subseteq\Psi[h]$.
By monotonicity, we get $f\subseteq\Phi[g]\subseteq\Phi[\Psi[h]]$.
By composing the two sequential computations, we get $f\subT h$.
\end{proof}

\begin{obs}
The least subTuring degree contains of exactly those functions with a partial computable extension.
\end{obs}

\subsection{Realizability theory}\label{sec:realizability-theory}
To emphasize the importance of subTuring degrees, as mentioned in Section \ref{sec:background}, this is not only related to pure computability theory, but also to realizability theory \cite{vOBook}.
The purpose of Section \ref{sec:realizability-theory} is to explain this connection, and since it will not be used in the remaining sections, readers who are only interested in pure computability theory may skip this section.

\begin{definition}
A {\em partial combinatory algebra} (pca) is a set equipped with a partial binary operation which is combinatory complete; see \cite{vOBook}.
For a partial function $f\pcolon\om\to\om$, we consider the pca $K_1[f]=(\om,\ast_f)$ of $f$-relative computability; that is, the partial binary operation $\ast_f\pcolon\om^2\to\om$ is defined as follows:
\[{e\ast_f n\downarrow}=m\iff {\Phi_e[f](n)\downarrow}=m,\]
where $\Phi_e$ is the $e$th partial computable function.
\end{definition}

Let ${\sf RT}(A)$ be the realizability topos induced by a pca $A$.

\begin{fact}[{\cite[Corollary 2.11]{FavO14}}]\label{fact:realizability-topos}
Every realizability subtopos of the effective topos is equivalent to one of the form ${\sf RT}(K_1[f])$ for some partial function $f$ on the natural numbers.

Moreover, there is an inclusion of toposes ${\sf RT}(K_1[f])\embed{\sf RT}(K_1[g])$ if and only if $g\subT f$.
\end{fact}

Assuming the former assertion, the latter assertion can also be derived from the correspondence between the subTuring degrees and some suborder of the poset of Lawvere-Tierney topologies; see also \cite{Kih21,Kih22}.
The important point is that this is based on subTuring reducibility, and not on nondeterministic Turing reducibility (enumeration reducibility).
In this way, the theory of subTuring degrees can be seen as a classification theory of the realizability subtoposes of the effective topos, or a classification theory of $K_1$-based pcas.

\section{Basic operations on subTuring degrees}\label{sec:basic}

\subsection{Lattice structure}
Let us examine the basic structure of the subTuring degrees.
Most degree structures studied in computability theory are upper semilattice, but they often do not have a meet (consider the Turing degrees, the many-one degrees, the truth-table degrees, the enumeration degrees, and so on).

Degree structures with meet are also known, but they are degree structures for problems with many solutions (e.g., the Medvedev degrees, the Muchnik degrees, and the Weihrauch degrees), and in the context of this article, they correspond to degree structures for multi-valued functions.
Surprisingly, the subTuring degree structure has a meet even though it deals with single-valued functions.

\begin{prop}\label{prop:subTuring-lattice}
The subTuring degrees form a lattice.
\end{prop}

\begin{proof}
The join $f\oplus g$ is defined as usual:
\[
(f\oplus g)(2n)=f(n);\qquad (f\oplus g)(2n+1)=g(n).
\]

The meet $f\cap g$ is defined as follows:
\[
(f\cap g)(d,e,n)=
\begin{cases}
\Phi_d[f](n)&\mbox{ if }{\Phi_d[f](n)\downarrow}={\Phi_e[g](n)\downarrow}\\
\uparrow&\mbox{ otherwise}.
\end{cases}
\]

For $f\cap g\subT f$, to compute $(f\cap g)(d,e,n)$, simulate $\Phi_d[f](n)$.
If $(d,e,n)\in{\rm dom}(f\cap g)$ then $\Phi_d[f](n)\downarrow$ by definition, so any queries made during the computation are contained within the domain of $f$.
The same applies to $f\cap g\subT g$.

To show that $f\cap g$ is the greatest lower bound of $f$ and $g$, assume $h\subT f,g$.
Then there exist $d$ and $e$ such that $h\subseteq\Phi_d[f],\Phi_e[g]$.
This means that $n\in{\rm dom}(h)$ implies ${\Phi_d[f](n)\downarrow}={\Phi_e[g](n)\downarrow}=h(n)$, so we have $(d,e,n)\in{\rm dom}(f\cap g)$ and $(f\cap g)(d,e,n)=h(n)$.
This shows $h\subT f\cap g$.
\end{proof}

Thus, the subTuring degrees have the property of being a lattice, which is very rare in degree theory.
A few other degree structures that are lattices, such as Medvedev and Weihrauch degrees, are known to be distributive lattices.
Is the subTuring lattice also distributive?
Quite interestingly, the subTuring lattice is non-distributive!
The result that it is a lattice, and that it is not distributive despite being a lattice, implies that it is not elementarily equivalent to other known degree structures.
The proof is a bit technical, so it is better to get used to the basic proof techniques first, so we leave the proof for later (Theorem \ref{thm:subTuring-nondistributive}).

\subsection{Completeness}

Next, let us discuss countable joins and meets.
There is no nontrivial countable supremum in several degree structures, including the Turing degrees.
Spector's proof uses the notion of an exact pair, which leads to the absence of a meet, but is not applicable to the subTuring degrees by Proposition \ref{prop:subTuring-lattice}.
Nevertheless, it is easy to prove the absence of nontrivial countable supremum in the subTuring degrees using a similar method.

\begin{prop}\label{prop:no-strictly-increasing}
No strictly increasing $\om$-sequence of subTuring degrees has a supremum.
\end{prop}

\begin{proof}
Assume that we are given a countable sequence $(g_n)_{n\in\om}$ of partial functions of increasing subTuring degrees.
Let us check that given any upper bound $h$ of $(g_n)_{n\in\om}$, we can construct another upper bound $f$ such that $h\not\subT f$.
For this purpose, we construct a strictly increasing sequence $(a_n)_{n\in\om}$ and guarantee that $f(a_n,x)=g_n(x)$.
Also, let $f(b,x)$ be undefined if $b\not=a_n$ for any $n\in\om$.
In this case, it is clear that $g_n\subT f$ for any $n\in\om$, so we want to guarantee $h\not\subT f$ by choosing $a_n$ appropriately.

We describe a strategy to ensure $h\not\subT f$ via $e$, i.e., $h\not\subseteq\Phi_e[f]$.
Assume that $(a_n)_{n\leq e}$ has already been constructed  and that $f(b,x)$ is determined for any $b\leq a_e$.
At stage $e$, perform the following actions:
Ask if the computation of $\Phi_e[f](n)$ for some input $n\in{\rm dom}(h)$ makes a query $(b,x)$ for some $b>a_e$ and $x\in\om$.

\medskip
\noindent
{\it Case (1):}
If yes, declare $f(b,x)\uparrow$.
In other words, put $a_{e+1}=b+1$ and proceed to the next stage $e+1$.
In this case, the computation of $\Phi_e[f](n)$ for some $n\in{\rm dom}(h)$ makes a query to $(b,x)\not\in{\rm dom}(f)$, which implies $\Phi_e[f](n)\uparrow$.
In particular, we have $h\not\subseteq\Phi_e[f]$.

\medskip
\noindent
{\it Case (2):}
If no, then do nothing and proceed to the next stage $e+1$.
In this case, only queries whose first coordinates are at most $a_e$ are asked to $f$ during the computation of $\Phi_e[f](n)$ for any $n\in{\rm dom}(h)$.
By our definition of $f$, using $g_0\oplus\dots \oplus g_e$ one can easily compute $f(b,x)$ for any $b\leq a_e$ and $x\in\om$.
In other words, $f\upto (a_e+1)\times\om\subT g_0\oplus\dots\oplus g_e\subT g_e$.
By assumption, we have $h\not\subT g_e$, so $h\not\subT f\upto (a_e+1)\times\om$.
However, the computation of $\Phi_e[f]$ is performed by accessing only to $f\upto (a_e+1)\times\om$, so computing $h$ must have failed somewhere.
Hence, we get $h\not\subseteq\Phi_e[f]$.
\end{proof}

By a somewhat dual argument, the absence of nontrivial countable infimum can also be easily proved.

\begin{prop}\label{prop:no-strictly-decreasing}
No strictly decreasing $\om$-sequence of subTuring degrees has an infimum.
\end{prop}

\begin{proof}
Assume that we are given a countable sequence $(g_n)_{n\in\om}$ of partial functions of decreasing subTuring degrees.
Let us check that given any lower bound $h$ of $(g_n)_{n\in\om}$, we can construct another lower bound $f\not\subT h$.
Now, instead of $g_n$, consider $g_n^\ast=g_0\cap\dots\cap g_n$.
For simplicity, we can assume that an input of $g_n^\ast$ is of the form $(e_0,\dots,e_n,x)$.
That is, $g_n^\ast(e_0,\dots,e_n,x)$ is the common value of $\Phi_{e_0}[g_0](x),\dots,\Phi_{e_n}[g_n](x)$ if it is defined.
In particular, we have ${\rm dom}(g_n^\ast)\subseteq \om^{n+2}$, so we can assume that the domains of $g_n^\ast$'s are pairwise disjoint.
We construct a strictly increasing sequence $(a_n)_{n\in\om}$ such that $a_n\in{\rm dom}(g_n^\ast)$ and define $f(a_n)=g_n^\ast(a_n)$.
Also, let $f(b)$ be undefined if $b\not=a_n$ for any $n\in\om$.

First, let us check $f\subT g_n^\ast$ no matter how we choose $(a_n)_{n\in\om}$.
For each $m\geq n$, if $a_m$ is of the form $\langle e_0,\dots,e_m,x\rangle$ and $f(a_m)\downarrow$, then the value of $f(a_m)$ is the common value of $\Phi_{e_0}[g_0](x),\dots,\Phi_{e_m}[g_m](x)$, which is especially the common value of $\Phi_{e_0}[g_0](x),\dots,\Phi_{e_n}[g_n](x)$.
Therefore, it follows that $f(a_m)=g_m^\ast(e_0,\dots,e_m,x)=g_n^\ast(e_0,\dots,e_n,x)$.
In this way, using $g_n^\ast$ one can compute $f$ except for finitely many values, so we obtain $f\subT g_n^\ast$.
Now we want to guarantee $f\not\subT h$ by choosing $a_n$ appropriately.

We describe a strategy to ensure $f\not\subT h$ via $e$, i.e., $f\not\subseteq\Phi_e[h]$.
Assume that $(a_n)_{n<e}$ has already been constructed.
At stage $e$, perform the following actions:
Find $a_{e}\in{\rm dom}(g_e)$ such that $a_{e}>a_{e-1}$ and $\Phi_e[h](a_{e})\not=g_e^\ast(a_{e})$.
Since $g_e^\ast\not\subT h$ by assumption, such $a_{e}$ must exist.
Then proceed to the next stage $e+1$.
Note that $a_{e}\in{\rm dom}(g_e)$ implies $a_{e}\in {\rm dom}(f)$ by definition.
Therefore, our strategy ensures $\Phi_e[h](a_e)\not=f(a_e)$, which implies $f\not\subseteq \Phi_e[h]$.
\end{proof}

\subsection{Jump Operator}\label{sec:jump-operator}

For Turing degrees, there is an important operator called the Turing jump, which is defined as the relative halting problem.
There is also the enumeration jump operator for the enumeration degrees, which is an extension of the Turing jump operator.
It is hoped, then, that the notion of jump for the subTuring degrees can also be introduced.
It is natural to require, for example, the following for the subTuring jump operator $J$.
\begin{enumerate}
\item Inflationary: $f\subT J(f)$
\item Monotone: $f\subT g\implies J(f)\subT J(g)$
\item Conservative: $f$ total $\implies$ $J(f)\equiv_{\rm subT} f'$
\item Strict: $J(f)\not\subT f$
\end{enumerate}

Here, $f'$ denotes the Turing jump of $f$.
However, currently, no operator $J$ has been found that satisfies all of these properties.
For example, the most obvious candidate for $J$ is probably the relative halting problem, as usual.
\[
K_0(f)(e)=
\begin{cases}
1&\mbox{ if }\Phi_e[f](e)\downarrow\\
0&\mbox{ if }\Phi_e[f](e)\uparrow
\end{cases}
\]

However, this is not appropriate for a jump operator in the following sense:

\begin{obs}
$K_0$ is not monotone.
\end{obs}

\begin{proof}
Observe $\chi_{{\rm dom}(f)}\subT K_0(f)$.
This is because if $\Phi_{d(n)}$ is a computation that asks the query $n$ and halts regardless of the solution, then being $n\in{\rm dom}(f)$ and $\Phi_{d(n)}^f(d(n))$ halting are equivalent, which is equivalent to $d(n)\in K_0(f)(d(n))$.

If $f$ is a total function, then $f\upto A\subT f$ for any $A$.
Let us take $A$ such that $\chi_A\not\subT K_0(f)$.
If $K_0$ is monotone, then we get $\chi_A\subT K_0(f\upto A)\subT K_0(f)$, which is a contradiction.
\end{proof}

However, we cannot give up the monotonicity necessary to guarantee that the jump operator is well-defined on the subTuring degrees.
Considering what went wrong, we just defined $K_0(f)$ as the relative halting problem, but we must not forget that in our relative computation, there are two kinds of reasons for ``not halting''.
That is, one is simply that a machine (a Turing functional $\Phi$) does not halt, and the other is that the query was made outside the domain of an oracle and we are waiting forever for the oracle's response.
Thus, it is reasonable to define the relative halting problem as follows:
\[
K(f)(e)=
\begin{cases}
1&\mbox{ if }\Phi_e[f](e)\downarrow\\
0&\mbox{ if }\Phi_e[f](e)\uparrow\mbox{ without making a query outside of ${\rm dom}(f)$,}\\
\uparrow&\mbox{ if }\Phi_e[f](e)\uparrow\mbox{ because of making a query outside of ${\rm dom}(f)$.}
\end{cases}
\]

This is similar to the way in which a feedback Turing machine \cite{AFL20} distinguishes between non-halting and freezing computations, but we have not yet figured out the formal correspondence.
This relative halting problem $K$ has the expected properties.
\begin{prop}
$K$ satisfies the following conditions:
\begin{enumerate}
\item Uniformly inflationary: $f\subT K(f)$ via a single index.
\item Uniformly monotone: $f\subT g\implies K(f)\subT K(g)$ uniformly.
\item Uniformly conservative: $f$ total $\implies$ $K(f)\equiv_{\rm subT} f'$ via a single index.
\end{enumerate}
\end{prop}

\begin{proof}
(1) One can easily find an index $i(n)$ such that $\Phi_{i(n)}[f](x)$ simulates $f(n)$ no matter what $f$ and $x$ are.
Then we have $f(n)=K(f)(i(n))$; hence $f\subT K(f)$ via an index of $\lambda n.i(n)$.

(2) Assume $f\subT g$ via $d$; that is, $f\subseteq\Phi_d[g]$.
If $e\in{\rm dom}(K(f))$, $Phi_e[f](e)$ makes queries only within the domain of $f$.
Therefore, the replies to these queries by $f$ and $\Phi_d[g]$ are identical.
That is, $\Phi_e[f](e)$ and $\Phi_e[\Phi_d[g]](e)$ perform the same computation, especially whether they halt or not.
One can effectively find an index $b(d,e)$ such that $\Phi_{b(d,e)}[g](x)$ simulates $\Phi_e[\Phi_d[g]](e)$ no matter what $g$ and $x$ are.
Then we get $K(f)(e)=K(g)(b(d,e))$.
Consequently, $f\subT g$ via $d$ implies $K(f)\subT K(g)$ via an index of $\lambda e.b(d,e)$.

(3) If $f$ is total then so is $K(f)$.
Thus, $K(f)=K_0(f)=f'$.
\end{proof}

From these properties, we can automatically derive the strictness of $K$ on hyperarithmetical functions.
We say that a partial function $f$ is hyperarithmetical if it extends to a total $\Delta^1_1$ function.

\begin{prop}
If $J$ is uniformly inflationary, uniformly monotone, and uniformly conservative, then $f<_{\rm subT}J(f)$ for any partial hyperarithmetical function $f$.
Indeed, $f\equiv_{\rm subT}J(f)$ implies $\emptyset^{(\alpha)}<_{\rm subT}f$ for any computable ordinal $\alpha$.
\end{prop}

\begin{proof}
Assume $J(f)\subT f$.
Let $\mathcal{O}$ be Kleene's system of ordinal notations \cite{ChYu15}, and let $a\in\mathcal{O}$.
We use $\emptyset^{(a)}$ to denote the iterated Turing jump of $\emptyset$ along the notation $a$.
If $a=2^b$ and $\emptyset^{(b)}\subT f$ via $p_b$ then by monotonicity and conservativity, we have $\emptyset^{(a)}\equiv_{\rm subT}J(\emptyset^{(b)})\subT J(f)\subT f$ via some index $p_a$.
By uniformity, $p_b\mapsto p_a$ is computable.
If $a=3\cdot 5^e$ and $\emptyset^{(\varphi_e(n))}\subT f$ via $p_{\varphi_e(n)}$ for any $n\in\om$, then whenever the sequence $(p_{\varphi_e(n)})_{n\in\om}$ is computable we have $\emptyset^{(a)}=\bigoplus_{n\in\om}\emptyset^{(\varphi_e(n))}\subT f$ via some index $p_a$, where $(p_{\varphi_e(n)})_{n\in\om}\mapsto p_a$ is computable.
Hence, by effective transfinite induction, we get $\emptyset^{(a)}\subT f$ for any $a\in\mathcal{O}$.
\end{proof}

Unfortunately, $K$ does not satisfy strictness in general.
In other words, $K$ has a fixed point. 

\begin{theorem}\label{thm:jump-fixed-point}
There exists a partial function $f$ such that $K(f)\equiv_{\rm subT}f$.
\end{theorem}

\begin{proof}
Let $\theta_e^\Sigma$ and $\theta_e^\Pi$ be the $e$th $\Sigma^1_1$ sentence and the $e$th $\Pi^1_1$ sentence, respectively.
Define $\psi$ as follows:
\[
\psi(a,b)=
\begin{cases}
1&\mbox{ if both $\theta_a^\Sigma$ and $\theta_b^\Pi$ are true}\\
0&\mbox{ if both $\theta_a^\Sigma$ and $\theta_b^\Pi$ are false}\\
\uparrow&\mbox{ otherwise}
\end{cases}
\]

We claim $K(\psi)\subT\psi$.
To see this, we simulate the computation of $\Phi_e[\psi](e)$ for given $e\in{\rm dom}(K(\psi))$.
The history of the computation of $\Phi_e[\psi](e)$ of length $s$ is a sequence $(i_0,q_0,a_0,i_1,q_1,a_1,\dots,i_s,q_s)$ such that,  for any $k<s$, $i_k=0$ and ${\Phi_e(e,a_0,\dots,a_k)\downarrow}=\langle i_{k+1},q_{k+1}\rangle$, and if $q_k=\langle u,v\rangle$ then
\begin{align}\label{equ:jump-fixedpoint}
\mbox{either ($\theta_u^\Sigma$ is true and $a_k=1$) or ($\theta_v^\Pi$ is false and $a_k=0$).}
\end{align}

This gives a $\Sigma^1_1$-description of being a history of some computation.
If $e\in{\rm dom}(K(\psi))$, then we must have $q_k=\langle u,v\rangle\in{\rm dom}(\psi)$, which means that the truth values of $\theta_u^\Sigma$ and $\theta_v^\Pi$ are the same.
Thus, the condition (\ref{equ:jump-fixedpoint}) can be changed to the following:
\begin{align}\label{equ:jump-fixedpoint2}
\mbox{either ($\theta_v^\Pi$ is true and $a_k=1$) or ($\theta_u^\Sigma$ is false and $a_k=0$).}
\end{align}

This gives a $\Pi^1_1$-description of being a history of some computation.
Now, $\Phi_e[\psi](e)$ halts if and only if there exists a sequence which is the history of $\Phi_e[\psi](e)$ of some length $s$ such that $i_s=1$.
This is described just by adding existential quantifiers on natural numbers, so this also has $\Sigma^1_1$-description $\theta_{c(e)}^\Sigma$ and $\Pi^1_1$-description $\theta_{d(e)}^\Pi$.
Therefore, for $e\in{\rm dom}(K(\psi))$, whether $\Phi_e[\psi](e)$ halts or not can be decided by looking at the value of $\psi(c(e),d(e))$.
This verifies the claim.
\end{proof}

\begin{question}
Does there exist a uniformly inflationary, uniformly monotone, uniformly conservative, strict operator on the subTuring degrees?
\end{question}

\section{Degree-theoretic properties}\label{sec:theorem}

\subsection{Density}

Many degree structures, including the Turing degrees, are not globally dense, even though they may have some local substructures which are dense.
Interestingly, in contrast to those degree structures, the structure of the subTuring degrees is globally dense.

\begin{theorem}\label{thm:subTuring-dense}
The subTuring degrees are dense; that is, if $g<_{\rm subT}f$ then there is $h$ such that $g<_{\rm subT}h<_{\rm subT}f$.
\end{theorem}

In order to prove this, let us introduce a very useful ``anti-cupping'' lemma that will be the source of many results.
In order to describe a statement, we need the following notion:
For an oracle $\alpha$, a set $A\subseteq\om$ is {\em $\alpha$-immune} if $A$ has no $\alpha$-c.e.~subset.

\begin{lemma}\label{lem:strong-anti-cupping}
For $\alpha,\beta,f\in\om^\om$, let $A\subseteq\om$ be an $(\alpha\oplus\beta\oplus f)$-immune set.
For any $g\pcolon\om\to\om$, if ${\rm dom}(g)$ is $\alpha$-c.e., and $g\subT (f\upto A)\oplus\beta$, then $g\subT\beta$.
\end{lemma}

\begin{proof}
Assume that $g\subT (f\upto A)\oplus\beta$ via a partial computable function $\Phi$.
By monotonicity, we have $g\subseteq\Phi[(f\upto A)\oplus\beta]\subseteq\Phi[f\oplus\beta]$, so there is no difference between the computation of $\Phi[(f\upto A)\oplus\beta]$ and $\Phi[f\oplus\beta]$ for input $n\in{\rm dom}(g)$.
Now, look at the queries $q_0,q_1,\dots$ made during the computation of $\Phi[f\oplus\beta](n)$, take out the even queries $2p_0,2p_1,\dots$ (that is, the queries asking to $f$), and put $Q_n=\{p_0,p_1,\dots\}$, which is the set of queries asking to $f$.
Note that $\bigcup_nQ_n$ is clearly $(f\oplus\beta)$-c.e.
Furthermore, since $g\subseteq\Phi[(f\upto A)\oplus\beta]$, the queries to $f$ that we make for input $n\in{\rm dom}(g)$ must be contained in ${\rm dom}(f\upto A)=A$; that is, if $n\in{\rm dom}(g)$, then $Q_n\subseteq A$.
Consider $Q=\bigcup\{Q_n:n\in{\rm dom}(g)\}$, which is an $(\alpha\oplus\beta\oplus f)$-c.e.~subset of $A$.
By immunity, $Q$ is finite.
Since this is all the queries to $f$ during the computation of $g$ by $\Phi$, we obtain $g\subseteq\Phi[(f\upto Q)\oplus\beta]$.
However, since $f\upto Q$ is a finite function, it is computable, so we get $g\subT\beta$ as desired.
\end{proof}

Lemma \ref{lem:strong-anti-cupping} suggests that all nonzero degrees have properties close to the so-called strong anticupping property.
Here, for an upper semilattice $L$, we say that an element $a\in L$ has the {\em strong anticupping property} if there exists $b\in L$ such that $a\leq b\lor c$ implies $a\leq c$ for any $c\in L$.
Note that $A$ depends on $\beta$ in Lemma \ref{lem:strong-anti-cupping}, so it does not show that strong anticupping property.
However, it is powerful enough to help prove various results discussed below.

A partial function $f\pcolon\om\to\om$ is {\em quasiminimal} if $f$ is non-computable, but any total function $g\subT f$ is computable.

\begin{prop}\label{prop:quasiminimal-subTuring}
Every nonzero subTuring degree bounds a quasiminimal subTuring degree.
\end{prop}

\begin{proof}
Let $f\pcolon\om\to\om$ be a non-computable function.
By Lemma \ref{lem:strong-anti-cupping}, if $A$ is $f$-immune then every total function $g\subT f\upto A$  is computable.
We also have $f\upto A\subT f$.
Therefore, it suffices to construct an $f$-immune set $A$ such that $f\upto A$ is not computable.

At stage $e$, perform the following actions:
\begin{enumerate}
\item Assume that $A\upto r_e$ has already been determined.
Search for $n>r_e$ such that either $\varphi_e(n)\uparrow$ or $\varphi_e(n)\not=f(n)$.
Note that such an $n$ always exists; otherwise ${\varphi_e(n)\downarrow}=f(n)$ for any $n>r_e$, which implies that $f$ is computable, a contradiction.
Choose such an $n$, declare $n\in A$, and put $u_{e}=n+1$.

\item 
Search for $n>u_e$ such that $n\in W_e^f$, where $W_e^f$ is the $e$th c.e.~set relative to $f$.
If $W_e^f$ is infinite, such an $n$ exists.
For such an $n$, declare $n\not\in A$, and put $r_{e+1}=n+1$.
If there is no such $n$, put $r_{e+1}=u_e$.
\end{enumerate}

By the action (1), for any $e$ we have $\varphi_e(n)\not=f(n)$ for some $n\in A$, so $f\upto A$ is non-computable.
By the action (2), the final obtained $A$ is $f$-immune.
Thus, by the argument above, $f\upto A$ is quasiminimal.
\end{proof}


Now let us show density of the subTuring degrees.

\begin{proof}[Proof (Theorem \ref{thm:subTuring-dense})]
Let $f,g\pcolon\om\to\om$ be such that $g<_{\rm subT} f$.
Assume that the domain of $g$ is $\alpha$-c.e.
As in the proof of Proposition \ref{prop:quasiminimal-subTuring}, one can easily construct an $(\alpha\oplus f)$-immune set $A$ such that $f\upto A\not\subT g$ (by using the assumption $f\not\subT g$).
Then we have $f\not\subT (f\upto A)\oplus g$; otherwise, by Lemma \ref{lem:strong-anti-cupping}, $f\subT (f\upto A)\oplus g$ implies $f\subT g$, a contradiction.
Consequently, we get $g<_{\rm subT}(f\upto A)\oplus g<_{\rm subT}f$.
\end{proof}

The above proof actually shows the existence of strong quasiminimal cover in any nonempty interval of subTuring degrees.
Here, we say that ${\bf a}$ is a {\em strong quasiminimal cover} of ${\bf b}$ if ${\bf b}<{\bf a}$ and, for any total degree ${\bf c}$, ${\bf c}\leq{\bf a}$ implies ${\bf c}\leq{\bf b}$.
This is because if $h\subT (f\upto A)\oplus g$ is total then $h\subT g$; hence $(f\upto A)\oplus g$ is a strong quasiminimal cover of $g$.

Many degree structures, including the Turing degrees, are known to be locally countable, i.e., any given degree can bound only countably many different degrees.
Surprisingly, the structure of subTuring degrees is not locally countable.
In fact, any nonzero subTuring degree can bound continuum many different degrees.

\begin{theorem}\label{thm:dense-continuum}
Any nonempty interval of subTuring degrees contains an antichain of continuum many different degrees.

In other words, if $g<_{\rm subT}f$ then there exists $\{h_\alpha\}_{\alpha<2^{\aleph_0}}$ such that $g<_{\rm subT}h_\alpha<_{\rm subT}f$  and $h_\alpha\not\subT h_\beta$ whenever $\alpha\not=\beta$.
\end{theorem}

\begin{proof}
The proof is almost the same as the proof of Proposition \ref{prop:quasiminimal-subTuring}.
Let $f,g\pcolon\om\to\om$ be such that $g<_{\rm subT} f$, and assume that the domain of $g$ is $\alpha$-c.e.
We construct a family $\{A_x\}_{x\in 2^\om}$ of continuum many subsets of $\om$ such that each $A_x$ is $(\alpha\oplus A_y\oplus f)$-immune for any $y\not=x$.
Moreover, we ensure that $f\upto A_x\not\subT g$ for any $x\in 2^\om$.
Assuming that we could construct such a family, consider $h_x=(f\upto A_x)\oplus g$ for any $x\in 2^\om$.
If $x\not=y$, by Lemma \ref{lem:strong-anti-cupping}, $h_y\subT h_x$ would imply $f\upto A_y\subT g$, which is impossible by our assumption on $\{A_x\}_{x\in 2^\om}$.
Hence, $\{h_x\}_{x\in 2^\om}$ gives an antichain of continuum many subTuring degrees.
Moreover, as in Theorem \ref{thm:subTuring-dense}, one can also see that $g<_{\rm subT}h_x<_{\rm subT}f$.
Therefore, it suffices to construct such a family $\{A_x\}_{x\in 2^\om}$.

At stage $e$, perform the following actions:
\begin{enumerate}
\item Assume that $A_\sigma\upto r_e$ has already been determined for each $\sigma\in 2^e$.
Search for $n\in{\rm dom}(f)$ with $n>r_e$ such that either $\Phi_e[g](n)\uparrow$ or $\Phi_e[g](n)\not=f(n)$.
Note that such an $n$ always exists; otherwise ${\Phi_e[g](n)\downarrow}=f(n)$ for any $n\in{\rm dom}(f)$ with $n>r_e$, which implies that $f\subT g$, a contradiction.
Choose such an $n$, declare $n\in A_\sigma$ for any $\sigma\in 2^e$, and put $u_{e}=n+1$.

\item 
There are $2^{e+1}$ substages in the strategy to guarantee relative immunity.
At substage $0$, put $u^0_e=u_e$.
At substage $t<2^{e-1}$, we focus on the $t$-th string $\sigma\fr i\in 2^{e+1}$.
Search for a finite string $\rho$ extending $A_\sigma\upto u_e^t$ and a number $n>u_e^t$ such that $n\in W_e^{\alpha\oplus \rho\oplus f}$.
If such $\rho$ and $n$ exist, then declare that $A_{\sigma\fr i}$ extends $\rho$, and protect $A_{\sigma\fr i}\upto |\rho|+1$.
Moreover, declare $n\not\in A_\tau$ for any string $\tau\not=\sigma\fr i$ of length $e+1$.
This means that we put $u_e^{t+1}=\max\{|\rho|,n\}+1$.
If there exist no such $\rho$ and $n$, put $A_{\sigma\fr i}=A_\sigma$ and $u_e^{t+1}=u_e$.
If $t=2^{e+1}-1$ then put $r_e=u_e^{t+1}$.
\end{enumerate}

Finally, we define $A_x=\bigcup_{\sigma\prec x}A_\sigma$ for any $x\in 2^\om$.
By the action (1), we get $f\upto A_\sigma\not\subseteq \Phi_e[g]$ for any $\sigma\in 2^e$, so $f\upto A_x\not\subseteq \Phi_e[g]$ since $A_\sigma\subseteq A_x$ for any $x\succ\sigma$; hence $f\upto A_x\not\subT g$ for any $x\in 2^\om$.
By the action (2), since $A_x$ extends $A_{x\upto e+1}\upto u_e^{t}$ for a suitable $t$, if $W_e^{\alpha\oplus A_x\oplus f}$ is infinite, we always find $\rho$ and $n>u_e^t$ such that $n\in W_e^{\alpha\oplus\rho\oplus f}$.
If $x\upto e\not= y\upto e$, then the action (2) ensures $n\not\in A_y$, so we obtain $W_e^{\alpha\oplus A_x\oplus f}\not\subseteq A_y$.
If $x\not=y$ then there is $k$ such that $x\upto k\not=y\upto k$, so given $e$, consider $d\geq k$ such that $W_d^{B}=W_e^{B}$ for any $B$ (for example, add meaningless lines to the program $e$ to pad its size).
By the above argument, we get $W_e^{\alpha\oplus A_x\oplus f}=W_d^{\alpha\oplus A_x\oplus f}\not\subseteq A_y$.
Consequently, $A_y$ is $(\alpha\oplus A_x\oplus f)$-immune.
\end{proof}

Note that $h_x$ in Theorem \ref{thm:dense-continuum} is a strong quasiminimal cover of $g$ for any $x\in 2^\om$.



\subsection{Jump inversion}
In Turing degree theory, the Friedberg jump inversion theorem states that for every total function $h\geq_T\emptyset'$ there is a total function $f$ such that $f'\equiv_Th$.
In enumeration degree theory, this has been improved to the quasiminimal jump inversion theorem.
Although the proof is different from the usual method, one can also prove the (quasiminimal) jump inversion theorem for the subTuring degrees.
Here, we adopt the jump operator $K$ introduced in Section \ref{sec:jump-operator}.

\begin{theorem}\label{thm:quasiminimal-jump-inversiona}
For any partial function $h\geq_{\rm subT}\emptyset'$ there exists a quasiminimal partial function $f$ such that $K(f)\eqsubT h\eqsubT f\oplus\emptyset'$.
\end{theorem}

\begin{proof}
It suffices to show that, for any partial function $h$, there exists a quasiminimal $f$ such that $K(f)\subT h\oplus\emptyset'\subT f\oplus\emptyset'$ since $\emptyset'\subT K(\emptyset)\subT K(f)$ by monotonicity and conservativity.
Let us construct a $\emptyset'$-computable increasing sequence $(a_k)_{k\in\om}$ of natural numbers.
This is the coding location of $h$ in $f$; that is, we define $f(a_k)=h(k)$ for each $k\in\om$ and let $f$ be undefined except on $\{a_k\}_{k\in\om}$.

Assume that $(a_k)_{k<e}$ has already been defined.
For each partial function $\sigma$ defined only on $\{a_k\}_{k<e}$, ask $\emptyset'$ if the computation of $\Phi_e[\sigma](e)$ makes a query greater than $a_{e-1}$.
If so, let $q_\sigma$ be the first such query plus $1$, or else let $q_\sigma=a_{e-1}+1$.
Similarly, ask $\emptyset'$ if the computation of $\Phi_e[\sigma](n)$ for some $n$ makes a query greater than $a_{e-1}$.
If so, let $r_\sigma$ be the first such query plus $1$, or else let $r_\sigma=a_{e-1}+1$.
Then let $a_e$ be the maximum value of these $q_\sigma$ and $r_\sigma$'s.

This completes the construction.
We claim that the computation of $\Phi_e[f](e)$ is determined by $\Phi_e[f\upto a_{e-1}+1](e)$.
Otherwise, this computation creates a query to $f$ greater than $a_{e-1}$.
The computation is performed by $\Phi_e[f\upto a_{e-1}+1](e)$ until just before the first such query $q$ is created.
Note that the domain of $\sigma:=f\upto a_{e-1}+1$ is $\{a_k\}_{k<e}$, so our construction ensures $a_{e-1}<q<q_\sigma\leq a_e$, which implies $q\not\in{\rm dom}(f)$.
Therefore, $\Phi_e[f](e)$ does not halt.

Observe that $(a_k)_{k\in\om}$ is $\emptyset'$-computable.
Thus, it is clear from the definition of $f$ that $h\subT f\oplus\emptyset'$ and $f\subT h\oplus\emptyset'$ hold.
Indeed, the graph of $f$ is computable in $h\oplus\emptyset'$.
We show that $K(f)\subT h\oplus\emptyset'$.
To compute $K(f)(e)$, first use $h\oplus\emptyset'$ to recover $(a_k)_{k< e}$ and the graph of $f\upto a_{e-1}+1$.
Ask $\emptyset'$ which of the following is true:

\begin{enumerate}
\item $\Phi_e[f\upto a_{e-1}+1](e)$ makes a query $\leq a_{e-1}$ that is different from any of $\{a_k\}_{k<e}$, or
makes a query that is greater than $a_{e-1}$.
\item $\Phi_e[f\upto a_{e-1}+1](e)$ halts.
\item Otherwise.
That is, $\Phi_e[f\upto a_{e-1}+1](e)$ does not halt even though it only makes queries in $\{a_k\}_{k<e}$.
\end{enumerate}

Since $(a_k)_{k< e}$ and $f\upto a_{e-1}+1$ are finite, this decision can be performed with the unrelativized halting problem $\emptyset'$.
By the above construction, if (1) holds, $K(f)(e)\uparrow$; if (2) holds, $K(f)(e)=1$; and if (3) holds, $K(f)(e)=0$.
This shows $K(f)\subT h\oplus\emptyset'$.

For quasiminimality, if $\Phi_e[f]$ is total, this computation can be simulated by $\Phi_e[f\upto a_{e-1}+1]$.
Otherwise, $\Phi_e[f\upto a_{e-1}+1](n)$ for some $n$ makes a query $q$ larger than $a_{e-1}$, but by the same argument as in Claim above, for $\sigma:=f\upto a_{e-1}+1$ is $\{a_k\}_{k<e}$, our construction ensures $a_{e-1}<q<r_\sigma\leq a_e$, which implies $q\not\in{\rm dom}(f)$.
Thus, $\Phi_e[f]$ is computable since this computation uses only finite information on $f$.
\end{proof}

\subsection{Minimal pair}
Next, let us analyze the properties of subTuring meets.
Turing degrees do not form a lattice, but there are several pairs of Turing degrees having a meet.
The following results lead to the conclusion that this Turing meet and the subTuring meet are never coincident.
In fact, every total degree is meet-irreducible within the total degrees:

\begin{theorem}\label{thm:no-minimal-pair}
For any total functions $f,g,h\colon\om\to\om$, if $f\cap g\subT h$ then either $f\subT h$ or $g\subT h$ holds.
\end{theorem}

\begin{proof}
We construct Turing functionals $\Psi$ and $\Delta$ fulfilling the following requirement:
\[R_e\colon (\forall f,g\in 2^\om)\;\Psi^f\cap\Delta^g\subseteq\Phi_e^h\implies f\subT h\mbox{ or }g\subT h.\]

Here $\Psi^f\cap\Delta^g$ is the intersection of $\Psi^f$ and $\Delta^g$ when identifying a function with its graph.
That is, if (and only if) both $\Psi^f(n)$ and $\Delta^g(n)$ converge to the same value, $(\Psi^f\cap\Delta^g)(n)$ halts at the value.
Obviously, $\Psi^f\cap\Delta^g\subT f\cap g$, so it is sufficient to satisfy this requirement.
Rewriting $R_e$ a little more concretely according to the actual strategy, we will ensure the following:
\[R_e\colon (\forall f,g\in 2^\om)\;[(\exists n)\;{\Psi^f(n)\downarrow}={\Delta^g(n)\downarrow}\not=\Phi_e^h(n)]\mbox{ or }f\subT h\mbox{ or }g\subT h.\]

\noindent
{\em Strategy:}
Let us describe the $R_e$-strategy.
The strategy must operate without knowing what $f$ and $g$ are and with all possibilities of $f,g\in 2^\om$ in mind.
This strategy consists of infinitely many substrategies.
The $(s,t)$-th substrategy proceeds as follows:
\begin{enumerate}
\item 
For any $f$ with $f(s)=i$, put $\Psi^f(e,s,t,i,j)=0$.
Similarly, for any $g$ with $g(t)=j$, put $\Delta^g(e,s,t,i,j)=0$.
\item 
Wait for $\Phi_e^h(e,s,t,i,j)\downarrow=0$ for some $i,j<2$.

If no such $i,j$ exists, then proceed to the next stage $e+1$.
In this case, $\Psi^f(e,s,t,f(s),g(t))=\Delta^g(e,s,t,f(s),g(t))=0\not=\Phi_e^h(e,s,t,f(s),g(t))$, so the requirement $R_e$ is satisfied.

If such $i$ and $j$ exist, fix the first pair $i,j$ found.
\item For any $f$ with $f(s)=1-i$, put $\Psi^f(e,s,t,i,j)=1$.
Similarly, for any $g$ with $g(t)=1-i$, put $\Delta^g(e,s,t,i,j)=1$.

Note that for any $f,g$ with $\langle f(s),g(t)\rangle=\langle 1-i,1-j\rangle$ then $\Psi^f(e,s,t,i,j)=\Delta^g(e,s,t,i,j)=1\not=0=\Phi_e^h(e,s,t,i,j)$, so the requirement $R_e$ is satisfied.
\item Wait for $\Phi^h_e(e,s,t,1-i,j)\downarrow=0$.
While waiting, invoke the $\langle s+1,t\rangle$-th substrategy.

If we wait forever, for any $f,g$ with $\langle f(s),g(t)\rangle=\langle 1-i,j\rangle$, then by (1), $\Psi^f(e,s,t,1-i,j)=\Delta^g(e,s,t,1-i,j)=0\not=\Phi_e^h(e,s,t,1-i,j)$, so the requirement $R_e$ is satisfied.
Thus, together with (3) we see that $R_e$ is satisfied for any value of $g(t)$ when $f(s)=1-i$.

\item Suppose that, after waiting for a while, $\Phi^h_e(e,s,t,1-i,j)\downarrow=0$ occurs.
While waiting, the $R_e$-strategy is still running, so suppose that the substrategy currently in operation is $(s',t)$ (the $(s,t)$-strategy will be abandoned when $t$ increases, so we may assume that the second coordinate remains at $t$).

Once we recognize that we reach (5), if $f(s)=i$, then put $\Psi^f(e,s,t,1-i,j)=1$; similarly, if $g(s)=1-j$, then put $\Delta^g(e,s,t,1-i,j)=1$.
Then stop the currently running substrategy and start a new substrategy $(s',t+1)$.

In this case, if $\langle f(s),g(s)\rangle=\langle i,1-j\rangle$, then $\Psi^f(e,s,t,1-i,j)=\Delta^g(e,s,t,1-i,j)=1\not=0=\Phi_e^h(e,s,t,1-i,j)$, so the requirement $R_e$ is satisfied.
Together with (3), we see that $R_e$ is satisfied for any value of $f(s)$ when $g(t)=1-j$.
However, since the fulfillment of the requirement in (4) with $\langle f(s),g(s)\rangle=\langle 1-i,j\rangle$ has been withdrawn, $R_e$ is not satisfied when $g(t)=j$.
\end{enumerate}

\noindent
{\em Verification:}
Let us check that the requirement $R_e$ is satisfied by this construction.

\medskip
\noindent
{\em Case 1:}
If we wait forever for (2) in either substrategy, the requirement $R_e$ is clearly satisfied.

Hereafter, we write $i$ and $j$ chosen in the state (2) of the $(s,t)$-substrategy as $i_{s,t}$ and $j_{s,t}$, respectively.
Note that $(s,t)\mapsto i_{s,t},j_{s,t}$ is $h$-computable since the substrategy proceeds in an $h$-computable process.

\medskip
\noindent
{\em Case 2:}
Consider the case where all substrategies pass through (3) but $t$ converges to a finite value.
In this case, for some $s_0$, the substrategies after $(s_0,t)$ do not reach (5) but wait at (4) and continue to activate the next substrategy.
As described in (4), the requirement $R_e$ is satisfied for any value of $g(t)$ when $f(s)=1-i_{s,t}$.
That is, the requirement $R_e$ is not yet satisfied only for $f(s)=i_{s,t}$ for any $s>s_0$.
However, in this case, as $s\mapsto i_{s,t}=f(s)$ is $h$-computable, one derives $f\subT h$.

\medskip
\noindent
{\em Case 3:}
Consider the case where $t$ diverges to infinity.
In this case, for each $t$, (5) occurs for $j=j_{s_t,t}$ defined by some substrategy $(s_t,t)$.
As described in (5), the requirement $R_e$ is satisfied for any value of $f(s_t)$ when $g(t)=1-j_{s_t,t}$.
That is, the requirement $R_e$ is not yet satisfied only for $g(t)=j_{s_t,t}$ for any $t$.
Note that one can $h$-computably recognize what $j_{s_t,t}$ is by waiting for the occurrence of (5), so $t\mapsto j_{s_t,t}$ is $h$-computable.
Thus, in this case, $g\subT h$ is derived and the requirement $R_e$ is satisfied.
\end{proof}

\begin{cor}
The meet of two incompatable total subTuring degrees is nontotal.
\end{cor}

This leads to the observation that, although the existence of a minimal pair is well known for Turing degrees, it does not form a minimal pair in the subTuring degrees.
Here, if a lattice $L$ has a bottom element $0$, a pair $(a,b)$ of elements in $L$ is called a {\em minimal pair} if $a,b>0$ and $a\land b=0$.

\begin{cor}\label{cor:no-minimal-pair}
There exists no minimal pair of total subTuring degrees.
\end{cor}


\subsection{Distributivity}
A lattice is distributive if $a\lor (b\land c)=(a\lor b)\land (a\lor c)$.
We now give a proof of non-distributivity of the subTuring degrees.

\begin{theorem}\label{thm:subTuring-nondistributive}
The subTuring lattice is not distributive.
\end{theorem}

\begin{proof}
We construct partial functions $f\pcolon\om\times 2\times 2\to\om$ and $g,h\pcolon\om\to 2$ such that $(f\oplus g)\cap (f\oplus h)\not\subT f\oplus(g\cap h)$.
It suffices to fulfill the following requirement:
\[R_e\colon (f\oplus g)\cap (f\oplus h)\not\subseteq\Phi_e[f\oplus (g\cap h)].\]

Given $e$, we need to find $c,d,x$ such that ${((f\oplus g)\cap (f\oplus h))(c,d,x)\downarrow}\not=\Phi_e[f\oplus (g\cap h)](c,d,x)$.
Given $n$, consider the following Turing functionals $\Psi_n$ and $\Gamma_n$:
\begin{align*}
\Psi_n[f\oplus g](x)=f(n,g(n),0)& &\Gamma_n[f\oplus h](x)=f(n,h(n),1)
\end{align*}

If we write indices of $\Psi_n$ and $\Gamma_n$ as $c_n$ and $d_n$, respectively, we actually guarantee the following requirement:
\[R_e\colon\exists n,x\;[{\Psi_n[f\oplus g](x)\downarrow}={\Gamma_n[f\oplus h](x)\downarrow}\not=\Phi_e[f\oplus (g\cap h)](c_n,d_n,x)].\]

By the definition of the meet $\cap$, it is easy to see that this modified requirement implies the first proposed requirement.

\medskip
\noindent
{\em Strategy:}
At stage $e$, operate the strategy to satisfy the requirement $R_e$.
Inductively, assume that we have constructed finite functions $f_e,g_e,h_e$ and an auxiliary parameter $n_e$.
We also assume that, for any $t<n_e$, $f_e(t,i,j)$, $g_e(t)$ and $h_e(t)$ have already been determined, and for any $t\geq n_e$, $f_e(t,i,j)$, $g_e(t)$ and $h_e(t)$ are undefined.
We refer to a triple of the form $\langle t,i,j\rangle$ as a level $t$ input.
Now let $c$ and $d$ be indices of $\Psi:=\Psi_{n_e}$ and $\Gamma:=\Gamma_{n_e}$, respectively.

This strategy consists of infinitely many rounds.
First, regarding the computation of $\Phi_e[f\oplus(g\cap h)](c,d,x)$, we have the following three possibilities.
\begin{enumerate}
\item[(1a)] $\Phi_e[f\oplus(g\cap h)](c,d,x)$ does not make a query, or makes a query only to $f$, which is of level $<n_e$.
\item[(1b)] $\Phi_e[f\oplus(g\cap h)](c,d,x)$ makes a level $\geq n_e$ query to $f$.
\item[(1c)] $\Phi_e[f\oplus(g\cap h)](c,d,x)$ makes a query to $g\cap h$.
\end{enumerate}

Here, in the $(k+1)^{st}$ round, we ask about the subsequent queries after the query taken up in the $k^{th}$ round.
As we will see later, no action is taken when advancing a round, always leaving room for the same action as in the first round.
Therefore, follow the discussion as if we were in the first round.

\medskip
\noindent
{\em Case (1a)}:
Put $g_{e+1}(n_e)=h_{e+1}(n_e)=0$, and set the level $n_e$ values of $f_{e+1}$ appropriately so that $f_{e+1}(n_e,g_{e+1}(n_e),0)\downarrow=f_{e+1}(n_e,h_{e+1}(n_e),1)$, whose value is different from $\Phi_e[f_e\oplus(g_e\cap h_e)](c,d,x)$.
Put $n_{e+1}=n_e+1$, and proceed to stage $e+1$.

\medskip
\noindent
{\em Verification}:
Since $\Phi_e[f_e\oplus(g_e\cap h_e)](c,d,x)$ only makes level $<n_e$ queries to $f$, the behavior of this computation coincides with $\Phi_e[f\oplus(g\cap h)](c,d,x)$.
By the action of the strategy, we have $\Psi[f\oplus g](x)=f(e,g(n_e),0)=f(e,h(n_e),1)=\Gamma[f\oplus h](x)\downarrow$, whose value is different from $\Phi_e[f\oplus(g\cap h)](c,d,x)$.
Thus, the requirement $R_e$ is satisfied.

\medskip
\noindent
{\em Case (1b)}:
There is a query to $f$ of the form $(m,i,j)$ for some $m\geq n_e$.
Declare $f_{e+1}(m,i,j)\uparrow$.
Put $g_{e+1}(n_e)=h_{e+1}(n_e)=1-i$.
Since $(m,i,j)\not=(n_e,g_{e+1}(n_e),0)$ and $(m,i,j)\not=(n_e,h_{e+1}(n_e),1)$, one can set the level $n_e$ values of $f_{e+1}$ so that $f_{e+1}(n_e,g_{e+1}(n_e),0)\downarrow=f_{e+1}(n_e,h_{e+1}(n_e),1)$.
Put $n_{e+1}=m+1$, and proceed to stage $e+1$.

\medskip
\noindent
{\em Verification}:
$\Phi_s[f\oplus (g\cap h)](d,e,x)$ makes a query $(m,i,j)\not\in{\rm dom}(f)$, so the computation does not halt.
We also have $\Psi[f\oplus g](x)=f(n_e,g(n_e),0)=f(n_e,h(n_e),1)=\Gamma[f\oplus h](x)\downarrow$.
Thus, the requirement $R_e$ is satisfied.

\medskip
\noindent
{\em Case (1c)}:
The computation makes a query $(a,b,y)$ to $g\cap h$.
By the definition of the meet $g\cap h$, the oracle's response to this query is the common value of $\Phi_a[g](y)$ and $\Phi_b[h](y)$ (if defined).
Let us now go into the computation processes of $\Phi_a[g](y)$ and $\Phi_b[h](y)$.
We have the following possibilities:
\begin{enumerate}
\item[(2a)] Either $\Phi_a[g](y)$ or $\Phi_b[h](y)$ does not halt, but makes only small queries less than $n_e$.
\item[(2b)] Either $\Phi_a[g](y)$ or $\Phi_b[h](y)$ halts after making only small queries less than $n_e$.
\item[(2c)] Either $\Phi_a[g](y)$ or $\Phi_b[h](y)$ makes a query larger than $n_e$.
\item[(2d)] None of the above; that is, both $\Phi_a[g](y)$ and $\Phi_b[h](y)$ make a query $n_e$.
\end{enumerate}

\medskip
\noindent
{\em Case (2a)}:
Put $g_{e+1}(n_e)=h_{e+1}(n_e)=0$, and set the values appropriately so that $f_{e+1}(n_e,g_{e+1}(n_e),0)\downarrow=f_{e+1}(n_e,h_{e+1}(n_e),1)$.
Put $n_{e+1}=n_e+1$, and proceed to stage $e+1$.

\medskip
\noindent
{\em Verification}:
Recall that we are in Case (1c), where $\Phi_e[f\oplus (g\cap h)](c,d,x)$ makes a query $(a,b,y)$ to $g\cap h$, and we then enter Case (2a), where $\Phi_a[g_e](c)$ or $\Phi_b[h_e](c)$ is undefined, so $(g\cap h)(a,b,y)$ is also undefined.
This means that $\Phi_e[f\oplus (g\cap h)](c,d,x)$ is accessing outside the domain of the oracle, so the computation does not halt.
By our action, we have $\Psi[f\oplus g](x)=f(n_e,g(n_e),0)=f(n_e,h(n_e),1)=\Gamma[f\oplus h](x)\downarrow$, so the requirement $R_e$ is satisfied.

\medskip
\noindent
{\em Case (2b)}:
In this case, the common value of $\Phi_a[g](c)$ and $\Phi_b[h](c)$ is determined by the current finite information $g_e$ or $h_e$.
Then, proceed to the next round $k+1$; that is, resume the computation of $\Phi_e[f\oplus (g\cap h)](c,d,x)$, and check the conditional branch (1) again for the subsequent queries of the query $(a,b,y)$ taken up this time.

\medskip
\noindent
{\em Case (2c)}:
If $\Phi_a$ makes a query $m>n_e$ to $g$, then declare $g(m)\uparrow$.
Do the same for $\Phi_b$ and $h$.
Put $g_{e+1}(n_e)=h_{e+1}(n_e)=0$, and set the values appropriately so that $f_{e+1}(n_e,g_{e+1}(n_e),0)\downarrow=f_{e+1}(n_e,h_{e+1}(n_e),1)$.
Put $n_{e+1}=n_e+1$, and proceed to stage $e+1$.

\medskip
\noindent
{\em Verification}:
Recall that we are in Case (1c), where $\Phi_e[f\oplus (g\cap h)](c,d,x)$ makes a query $(a,b,y)$ to $g\cap h$, and we then enter Case (2c), where $\Phi_a$ makes a query $m>n_e$ to $g$.
Then our action ensures $m\not\in{\rm dom}(g)$, so we have $\Phi_a[g](c)\uparrow$.
Do the same for the symmetric case.
If $\Phi_a[g_e](c)$ or $\Phi_b[h_e](c)$ is undefined, then $(g\cap h)(a,b,y)$ is also undefined, which means that $\Phi_e[f\oplus (g\cap h)](c,d,x)$ is accessing outside the domain of the oracle, so the computation does not halt.
By our action, we have $\Psi[f\oplus g](x)=f(n_e,g(n_e),0)=f(n_e,h(n_e),1)=\Gamma[f\oplus h](x)\downarrow$, so the requirement $R_e$ is satisfied.

\medskip
\noindent
{\em Case (2d)}:
In this case, both $\Phi_a[g](y)$ and $\Phi_b[h](y)$ make only queries less than or equal to $n_e$.
Therefore, depending on the values of $g(n_e)$ and $h(n_e)$, at most four different computations appear.
Looking at these four computations, we have the following possibilities:
\begin{enumerate}
\item[(3a)] Either $\Phi_a[g](y)$ or $\Phi_b[h](y)$ may not halt, depending on the value of $g(n_e)$ or $h(n_e)$.
\item[(3b)] Either $\Phi_a[g](y)$ or $\Phi_b[h](y)$ may make a query greater than $n_e$, depending on the value of $g(n_e)$ or $h(n_e)$.
\item[(3c)] $\Phi_a[g](y)$ or/and $\Phi_b[h](y)$ have the possibility of halting computations of different outputs, depending on the value of $g(n_e)$ or/and $h(n_e)$.
\item[(3d)] Otherwise, that is, no matter what values $g(n_e)$ and $h(n_e)$ are set to, $\Phi_a[g](y)$ and $\Phi_b[h](y)$ both halts and return the same output, making only queries less than or equal to $n_e$.
\end{enumerate}

\medskip
\noindent
{\em Case (3a)}:
For instance, if setting $g_{e+1}(n_e)=i$ forces $\Phi_a[g_{e+1}](y)$ not to halt, then set $h_{e+1}(n_e)$ to any value.
Do the same for the symmetric case.
Set the values appropriately so that $f_{e+1}(n_e,g_{e+1}(n_e),0)\downarrow=f_{e+1}(n_e,h_{e+1}(n_e),1)$.
Put $n_{e+1}=n_e+1$, and proceed to stage $e+1$.

\medskip
\noindent
{\em Case (3b)}:
For instance, if setting $g_{e+1}(n_e)=i$ forces $\Phi_a[g_{e+1}](y)$ to make a query $m>n_e$, then declare $g_{e+1}(m)\uparrow$, and set $h_{e+1}(n_s)$ to any value.
Do the same for the symmetric case.
Set the values appropriately so that $f_{e+1}(n_e,g_{e+1}(n_e),0)\downarrow=f_{e+1}(n_e,h_{e+1}(n_e),1)$.
Put $n_{e+1}=n_e+1$, and proceed to stage $e+1$.

\medskip
\noindent
{\em Case (3c)}:
If setting $g_{e+1}(n_e)=i$ and $h_{e+1}(n_e)=j$ force the condition $\Phi_a[g](y)\not=\Phi_b[h](y)$, then set the values appropriately so that $f_{e+1}(n_e,g_{e+1}(n_e),0)\downarrow=f_{e+1}(n_e,h_{e+1}(n_e),1)$.
Put $n_{e+1}=n_e+1$, and proceed to stage $e+1$.

Otherwise, put $g^i_{e+1}(n_e)=i$ for each $i<2$, and assume $\Phi_a[g^0_{e+1}](y)\not=\Phi_a[g^1_{e+1}](y)$.
Then put $h_{e+1}(n_e)=0$.
Assuming that (3a) does not hold, $\Phi_b[h_{e+1}](y)$ must halt, so $\Phi_b[h_{e+1}](y)\not=\Phi_a[g^i_{e+1}](y)$ for some $i<2$.
Then put $g_{e+1}=g^i_{e+1}$ for such $i$.
Do the same for the symmetric case.
Set the values appropriately so that $f_{e+1}(n_e,g_{e+1}(n_e),0)\downarrow=f(n_e,h_{e+1}(n_e),1)$.
Put $n_{e+1}=n_e+1$, and proceed to stage $e+1$.

\medskip
\noindent
{\em Verification}:
The requirement $R_e$ is satisfied in the Case (3a) for the same reason as (2a), and the Case (3b) for the same reason as (2c).
The same reasoning can be applied to the Case (3c), since we ensure $\Phi_a[g_{e+1}](y)\not=\Phi_b[h_{e+1}](y)$, which implies $(a,b,y)\not\in{\rm dom}(g\cap h)$.

\medskip
\noindent
{\em Case (3d)}:
In this case, the common value of $\Phi_a[g](y)$ and $\Phi_b[h](y)$ is determined by the current finite information $g_e$ or $h_e$ (since the computations ask only for values of $g\upto n_e+1$ and $h\upto n_e+1$ by $\neg$(2c), and the results of these computations do not depend on the values of $g(n_e)$ and $h(n_e)$).
Then, proceed to the next round $k+1$; that is, resume the computation of $\Phi_e[f\oplus (g\cap h)](c,d,x)$, and check the conditional branch (1) again for the subsequent queries of the query $(a,b,y)$ taken up this time.

\medskip
\noindent
{\em Verification}:
The next round proceeds only if the strategy reaches (2b) or (3d).
In this case, the response to the query $(a,b,y)$ made at (1c) can be computed only from the information of $g_e=g\upto n_e$ and $h_e=h\upto n_e$.
Moreover, the strategy does not take any action at (2b) or (3c), so nothing interferes with the next round of operation.
Also, if the strategy arrives at a state other than (2b) or (3d) at some round, then the requirement $R_e$ is satisfied there.

\medskip
\noindent
{\em Infinite rounds}:
If the strategy keeps arriving only at (2b) or (3d), it means that the round keeps going infinitely, and this also means, in particular that (1c) is passed through infinitely many times.
Then, the computation of $\Phi_e[f\oplus (g\cap h)](c,d,x)$ continues to make queries to $g\cap h$ infinitely many times, so $\Phi_e[f\oplus (g\cap h)](c,d,x)$ does not halt.
In this case, put $g_{e+1}(n_e)=h_{e+1}(n_e)=0$, and set the values appropriately so that $f_{e+1}(n_e,g_{e+1}(n_e),0)\downarrow=f_{e+1}(n_e,h_{e+1}(n_e),1)$.
Put $n_{e+1}=n_e+1$, and proceed to stage $e+1$.
Then the requirement $R_e$ is satisfied for the same reason as (1a).
\end{proof}

A lattice $L$ is {\em modular} if, for any $a,b,c\in L$, $a\leq b$ implies $a\lor(b\land c)=b\land (a\lor c)$.
Obviously, a distributive lattice is always modular.
By making a small modification to the proof of Theorem \ref{thm:subTuring-nondistributive}, it is actually possible to prove that the subTuring lattice is not modular.
\begin{theorem}\label{thm:subTuring-nonmodular}
The subTuring lattice is not modular.
\end{theorem}

\begin{proof}[Proof (Sketch)]
We construct partial functions $f\pcolon\om\times 2\times 2\to\om$ and $g,h\pcolon\om\to 2$ such that $(f\oplus g)\cap (f\oplus h)\not\subT f\oplus(g\cap h)$.
It suffices to fulfill the following requirement:
\[R_e\colon (f\oplus g)\cap (f\oplus h)\not\subseteq\Phi_e[f\oplus (g\cap (f\oplus h))].\]

Given $e$, we need to find $c,d,x$ such that ${((f\oplus g)\cap (f\oplus h))(c,d,x)\downarrow}\not=\Phi_e[f\oplus (g\cap h)](c,d,x)$.
For the construction, the description becomes a little more complicated, but the idea of the proof is exactly the same as Theorem \ref{thm:subTuring-nondistributive}.
For example, item (1c) is modified as follows.
\begin{enumerate}
\item[(1c)]
$\Phi_e(f\oplus (g\cap (f\oplus h)))$ makes a query to $g\cap(f\oplus h)$.
\end{enumerate}

If this happens, assume that the computation makes a query $(a,b,y)$ to $g\cap (f\oplus h)$.
Then, go into the computation processes of $\Phi_a[g](y)$ and $\Phi_b[f\oplus h](y)$.
The discussion that follows is exactly the same as Theorem \ref{thm:subTuring-nondistributive}, but for example, item (2c) is modified as follows.
\begin{enumerate}
\item[(2c)] The computation of $\Phi_a[g](y)$ or $\Phi_b[f\oplus h](y)$ makes either a query to $g$ or $h$ which is larger than $n_e$, or a query to $f$ which is of level $\geq n_e$.
\end{enumerate}

Note that a query to $f$ may be at level $n_e$.
In this case, the requirement is satisfied by exactly the same action as in Case (1b).
If there is a query to $g$ or $h$ that is greater than $n_e$, the situation is the same as Case (2c) in Theorem \ref{thm:subTuring-nondistributive}.

In Case (2d), both $\Phi_a[g](y)$ and $\Phi_b[f\oplus h](y)$ can only make queries to $g$ and $h$ which are less than or equal to $n_e$ and queries to $f$ of level less than $n_e$.
Therefore, the situation is exactly the same as Case (2d) in Theorem \ref{thm:subTuring-nondistributive}, so the exact same argument applies.
\end{proof}

While the structure is not a distributive lattice, we can nonetheless find incomparable subTuring degrees that satisfy the distributive law:

\begin{prop}
There are partial functions $f_0,f_1,f_2\pcolon\om\to\om$ such that their subTuring degrees are pairwise incomparable, and $f_0\oplus(f_1\cap f_2)\eqsubT(f_0 \oplus f_1)\cap(f_0\oplus f_2)$.

Similarly, there are partial functions $e_0,e_1,e_2\pcolon\om\to\om$ such that their subTuring degrees are pairwise incomparable, and ${e_0}\cap({e_1}\oplus{e_2})=({e_0}\cap{e_1})\oplus({e_0}\cap{e_2})$.

\end{prop}
\begin{proof}
We claim that given any subTuring incomparable partial functions $g,h$ such that $g\cap h\not\equiv_{\rm subT} \emptyset$, we can find sets $A$ and $B$ such that 
\begin{enumerate}
\item $g\upto A\nleq_{\rm subT} h$ and $h\upto B\nleq_{\rm subT} g$.
\item $g \nleq_{\rm subT} (g\upto A)\oplus (h\upto B)$ and $ h \nleq_{\rm subT} (g\upto A)\oplus (h\upto B)$. 
\item $g\cap h \nleq_{\rm subT} g\upto A$ and $g\cap h \nleq_{\rm subT} h\upto B$.
\end{enumerate}

If these conditions are satisfied, we can take ${f_0}= g\cap h $, ${f_1}=  g\upto A $ and ${f_2}=  h\upto B $ and take ${ e_0}= { f_1}\oplus { f_2}$, ${ e_1}=  g $ and ${ e_2}=  h $.
The condition (1) implies $f_1,f_2\not\subT f_0$, and the condition (3) implies that $f_0\not\subT f_1,f_2$ and $f_1$ is incomparable with $f_2$.
Hence, $f_0,f_1,f_2$ are pairwise subTuring incomparable.
The condition (2) implies $e_1,e_2\not\subT e_0$, and the condition (1) implies that $e_0\not\subT e_1,e_2$ and $e_1$ is incomparable with $e_2$.
Hence, $e_0,e_1,e_2$ are pairwise subTuring incomparable.

Now, we have $f_1\cap f_2=(g\upto A)\cap (h\upto B)\subT g\cap h=f_0$, so $f_0\oplus (f_1\cap f_2)\eqsubT f_0$.
Moreover, $f_0\subT f_0\oplus f_1=(g\cap h)\oplus(g\upto A)\subT g$, and similarly $f_0\subT f_0\oplus f_2\subT h$.
Hence, $f_0\subT(f_0\oplus f_1)\cap(f_0\oplus f_2)\subT g\cap h=f_0$.
This shows $f_0\oplus(f_1\cap f_2)\eqsubT f_0\eqsubT(f_0 \oplus f_1)\cap(f_0\oplus f_2)$.

Next, we have $e_0=(g\upto A)\oplus(h\upto B)\subT g\oplus h=e_1\oplus e_2$, so $e_0\cap(e_1\oplus e_2)\eqsubT e_0$.
Moreover, $g\upto A\subT((g\upto A)\oplus(h\upto B))\cap g=e_0\cap e_1\subT e_0$, and similarly, $h\upto B\subT e_0\cap e_2\subT e_0$.
Hence, $e_0=(g\upto A)\oplus (h\upto B)\subT (e_0\cap e_1)\oplus(e_0\cap e_2)\subT e_0$.
This shows ${e_0}\cap({e_1}\oplus{e_2})\eqsubT e_0\eqsubT ({e_0}\cap{e_1})\oplus({e_0}\cap{e_2})$.

In order to show the claim, we need to construct $A\upto s$ and $B\upto s$ for each $s$; assume that these have been decided for $s$, and our next step is to satisfy a requirement of type (1), (2) or (3) above.

To satisfy (1) we pick $n>s$ such that $n\in {\rm dom}(g)$ and $g(n)\neq\Phi[h](n)$, where $\Phi$ is the next Turing functional we are diagonalizing against. This $n$ must exist lest $g\leq_{\rm subT} h$. We then set $A\upto n+1=\left(A\upto s\right)\cup\{n\}$, which will ensure that $g\upto A\not\subseteq \Phi[h]$. We act similarly to make $h\upto B\not\subseteq \Phi[g]$.

To satisfy (2) we check if there is some $x\in{\rm dom}(g)$ such that the computation $\Phi[\left(g\upto \left(A\upto s\right)\right)\oplus \left(h\upto \left(B\upto s\right)\right)](x)$ makes a query to the oracle $g(n)$ or $h(n)$ for some $n>s$.
If no, then $g\subseteq \Phi[{\left(g\upto A\right)\oplus \left(h\upto B\right)}]$ implies that $g\subseteq \varphi$ for some partial computable function simulating $\Phi$ with the given finite oracle, which is impossible. If the answer is yes, witnessed by some least $n>s$, we take $A\upto n+1=\left(A\upto s\right)$ and $B\upto n+1=\left(B\upto s\right)$, in particular setting $n\not\in A\cup B$, and therefore ensuring $g\not\subseteq \Phi[{\left(g\upto A\right)\oplus \left(h\upto B\right)}]$. We do similarly to make $h\not\subseteq \Phi[{\left(g\upto A\right)\oplus \left(h\upto B\right)}]$.

To satisfy (3) we check if there is some $x\in {\rm dom}(g\cap h)$ such that the computation $\Phi[{g\upto (A\upto s)}](x)$ queries $g(n)$ for some $n>s$. If not, then just like in (2) above, we have that $g\cap h\not\subseteq\Phi[{g\upto A}]$, unless $g\cap h\equiv_{\rm subT}\emptyset$. If yes, then we just keep $n\not\in A$ by taking  $A\upto n+1=\left(A\upto s\right)$ which will ensure that $g\cap h\not\subseteq \Phi[{g\upto A}]$. We proceed similarly to make $g\cap h\not\subseteq \Phi[{h\upto B}]$. In fact, note that we can replace $g\cap h$ by any function which is  $\not\equiv_{\rm subT} \emptyset$.
\end{proof}

\subsection{Irreducibility}

For a lattice $L$, an element $a\in L$ is {\em join-irreducible} if, for any $b,c\in L$, $a=b\lor c$ implies either $a=b$ or $a=c$.

\begin{theorem}\label{thm:subTuring-join-irreducible}
There exists a nonzero join-irreducible subTuring degree.
\end{theorem}

\begin{proof}
We construct a noncomputable partial function $f$ such that for any $g,h\subT f$ if $f\subT g\oplus h$ then either $f\subT g$ or $f\subT h$ holds.
It suffices to fulfill the following requirements:
\begin{align*}
P_e\colon&\ f\not\subseteq\varphi_e.\\
R_{e,k,j}\colon&\ \mbox{If $g\subseteq\Phi_e[f]$ and $h\subseteq\Phi_k[f]$ and $f\subseteq\Phi_j[g\oplus h]$}\\
&\mbox{ then $f\subT g$ or $f\subT h$.}
\end{align*}

At the beginning of the stage $s$, assume that a computable infinite set $I_s\subseteq\om$ has been constructed, and $f\upto\min I_s$ has been determined.
That is, $f\upto\min I_s$ is guaranteed not to change after stage $s$.

\medskip
\noindent
{\em $P_e$-strategey:}
Assume $s=2e$.
For $n=\min I_s$, one can choose the value of $f(n)$ so that $f(n)\downarrow\not=\varphi_e(n)$.
Put $I_{s+1}=\{x\in I_s:x>n\}$.
Then, the values of $f$ up to $n$ will be protected.
The requirement $P_e$ is clearly satisfied.

\medskip
\noindent
{\em $R_{e,k,j}$-strategy:}
Assume $s=2\langle e,k,j\rangle+1$.
This strategy consists of infinitely many substrategies.
The $t^{th}$ substrategy works with the $t$-th smallest element $n_t$ of $I_s$.
In particular, $n_0=\min I_s$, and $f\upto n_0$ has already been determined.
As long as we focus only on this one substrategy, it is somewhat similar to the strategy of Theorem \ref{thm:subTuring-nondistributive}.
In particular, the $t^{th}$ substrategy itself also consists of infinitely many rounds.
First, regarding the computation of $\Phi_j[{g\oplus h}](n_t)$, we have the following possibilities.
\begin{enumerate}
\item[(1a)] $\Phi_j[g\oplus h](n_t)$ is undefined without making a query.
\item[(1b)] $\Phi_j[g\oplus h](n_t)$ halts without making a query.
\item[(1c)] $\Phi_j[g\oplus h](n_t)$ makes a query for $g$ or $h$.
\end{enumerate}

Here, in the $k^{th}$ round, we ask about the $k^{th}$ query during the computation of $\Phi_j[g\oplus h](n_t)$, but the details will be described later, so at first, follow the discussion as if we were in the first round.

\medskip
\noindent
{\em Case (1a):}
In this case, by setting $f(n_t)\downarrow$ one can guarantee that $f(n_t)\not=\Phi_j[g\oplus h](n_t)$.
Here, if it is not the first round, set $f(n_t)\downarrow$ to be a value consistent with the current assumption on $f(n_t)$ (made by the previous rounds).
Put $I_{s+1}=\{x\in I_s:x>n_t\}$, and proceed to stage $s+1$.

\medskip
\noindent
{\em Case (1b):}
If we are in the first round, the computation halts without ever creating a query, so the value of $\Phi_j[f\oplus g](n_t)$ is constant regardless of what oracle $f\oplus g$ is.
In this case, by setting $f(n_t)\downarrow$ to be an appropriate value, we can guarantee that $f(n_t)\not=\Phi_j[g\oplus h](n_t)$.
Put $I_{s+1}=\{x\in I_s:x>n_t\}$ and proceed to stage $s+1$.

If the round has already progressed, there may already be a restriction on the value of $f(n_t)$, which may prevent diagonalization, so the action in this case is discussed in detail below.

\medskip
\noindent
{\em Case (1c):}
Since the argument is symmetric, assume that the computation have made a query to $g$.
If the assumption $g\subseteq\Phi_e[f]$ is satisfied, then the solution of the query $m$ to $g$ should be obtained from the solution of $\Phi_e[f](m)$.
In this case, we have the following possibilities.
\begin{enumerate}
\item[(2a)] $\Phi_e[f](m)$ does not halt, but makes only small queries less than $n_0$.
\item[(2b)] $\Phi_e[f](m)$ halts after making only small queries less than $n_0$.
\item[(2c)] $\Phi_e[f](m)$ makes a query $n'\geq n_0$ with $n'\not=n_t$.
\item[(2d)] None of the above; that is, $\Phi_e[f](m)$ makes a query $n_t$.
\end{enumerate}

\medskip
\noindent
{\em Case (2a):}
In this case, $m\not\in{\rm dom}(g)$ is ensured, so $\Phi_j[{g\oplus h}](n_t)\uparrow$.
As in the Case (1a), by setting $f(n_t)\downarrow$ one can guarantee that $f(n_t)\not=\Phi_j[g\oplus h](n_t)$.
Here, set $f(n_t)\downarrow$ to be a value consistent with the current assumption on $f(n_t)$ (made by the previous rounds).
Put $I_{s+1}=\{x\in I_s:x>n_t\}$ and proceed to stage $s+1$.

\medskip
\noindent
{\em Case (2b):}
It means that the solution to the query $m$ for $g$ has already been determined from the current information in $f\upto n_0$.
In this case, proceed to round $k+1$ of the $t^{th}$ substrategy.

\medskip
\noindent
{\em Case (2c):}
By setting $f(n')\uparrow$ one can force $g(m)\uparrow$, which also forces $\Phi_j[{g\oplus h}](n_t)\uparrow$.
Since $n'\not=n_t$, as in the Case (1a), one can set $f(n_t)\downarrow$ to guarantee that $f(n_t)\not=\Phi_j[g\oplus h](n_t)$.
Here, set $f(n_t)\downarrow$ to be a value consistent with the current assumption on $f(n_t)$ (made by the previous rounds).
Put $I_{s+1}=\{x\in I_s:x>n_t,n'\}$ and proceed to stage $s+1$.

\medskip
\noindent
{\em Case (2d):}
In this case, the computation may depend on the value of $f(n_t)\in\{0,1\}$, so the case is divided again.
\begin{enumerate}
\item[(3a)] $\Phi_e[f](m)$ may make a query $n'\geq n_0$ with $n'\not=n_t$, depending on the value of $f(n_t)$.
\item[(3b)] $\Phi_e[f](m)$ may not halt, depending on the value of $f(n_t)$.
\item[(3c)] Otherwise, that is, $\Phi_e[f](m)$, for any value of $f(n_t)$, halts after only making a query $n'<n_0$ or $n'=n_t$.\end{enumerate}

Here, if the value of $f(n_t)$ has already been determined, we only look at that value.

\medskip
\noindent
{\em Cases (3a) and (3b):}
Set $f(n_t)$ to the corresponding value.
In this case, we perform the same action as in the Cases (2c) and (2a), respectively, and then proceed to stage $s+1$.

\medskip
\noindent
{\em Case (3c):}
In this case, proceed to round $k+1$ of the $t^{th}$ substrategy.
However, in this case, the computation of $\Phi_e[f](m)$ may depend on the value of $f(n_t)$, so we perform the round $k+1$ when $f(n_t)=0$ and when $f(n_t)=1$ in parallel.

\medskip
\noindent
{\em Next rounds:}
The next round proceeds only if the substrategy reaches (2b) or (3c).
If the substrategy keeps arriving only at (2b) or (3c) for some value of $f(n_t)$, it means that the round keeps going infinitely, and this also means, in particular that (1c) is passed through infinitely many times.
In this case, the computation of $\Phi_j[g\oplus h](n_t)$ continues to make queries infinitely many times, so $\Phi_j[g\oplus h](n_t)$ does not halt.
Thus, by setting $f(n_t)\downarrow$ one can guarantee that $f(n_t)\not=\Phi_j[g\oplus h](n_t)$.
Here, set $f(n_t)\downarrow$ to be a value consistent with the current assumption on $f(n_t)$.
Put $I_{s+1}=\{x\in I_s:x>n_t\}$, and proceed to stage $s+1$.

\medskip

Note that if, for some value of $f(n_t)$, the substrategy eventually reaches a state other than (1b), (2b), or (3c), then we proceed to the next stage, after satisfying the requirement.
Also, as we saw above, if we only arrive at (2b) or (3c) for some value of $f(n_t)$, we also satisfy the requirement and proceed to the next stage.
The only remaining case is when we arrive at (1b) no matter how we choose the value of $f(n_t)$.
It is necessary to describe the action at (1b) after the round has progressed.

\medskip
\noindent
{\em Case (1b):}
After the round has progressed, the computation may be asking for the value of $f(n_t)$.
At this time, there are two different possibilities in the Case (1b).
\begin{enumerate}
\item[(4a)] The value of $\Phi_j[g\oplus h](n_t)$ does not depend on whether $f(n_t)=0$ or $f(n_t)=1$.
\item[(4b)] The value of $\Phi_j[g\oplus h](n_t)$ depends on whether $f(n_t)=0$ or $f(n_t)=1$.
\end{enumerate}

\medskip
\noindent
{\em Case (4a):}
In this case, one can choose the value of $f(n_t)$ so that $f(n_t)\not=\Phi_e[g\oplus h](n_t)$.
Put $I_{s+1}=\{x\in I_s:x>n_t\}$, and proceed to stage $s+1$.

\medskip
\noindent
{\em Case (4b):}
In this case, start running the $(t+1)^{st}$ substrategy.

\medskip

Let us analyze the situation when (4b) occurs.
First, note that the behavior of each substrategy does not depend on the behavior of the other substrategies.
In other words, since the behavior of the $t^{th}$ substrategy depends only on the values of $f\upto n_0$ and $f(n_t)$, in fact, one can operate all substrategies in parallel.
Furthermore, if (4b) occurs, then the computation $\Phi_j[g\oplus h](n_t)$ only makes queries $n\in{\rm dom}(f)$ with $n<n_0$ or $n=n_t$.
Now, for each $i<2$, let $f_i$ be an extension of $f\upto n_0$ such that $f_i(n_t)=i$.
Dependently, the values of $g$ and $h$ also change, so that $\hat{g}_i=\Phi_e[f_i]$ and $\hat{h}_i=\Phi_k[f_i]$.

Note that (4b) occurs because the solution of a query differs somewhere in the $\Phi_j$-computations for $f(n_t)=0$ and for $f(n_t)=1$.
That is, for the queries $m_0,m_1,\dots$ during the computation of $\Phi_j[\hat{g}_0\oplus\hat{h}_0](n_t)$ and the queries $m_0',m_1',\dots$ during the computation of $\Phi_j[\hat{g}_1\oplus\hat{h}_1](n_t)$, since we are running the same computation $\Phi_j$, these queries must be the same until the very first difference on the solutions occurs.
That is, there exists some $k(t)$ such that $m_i=m_i'$ for each $i\leq k(t)$, but the solutions for $m_{k(t)}=m_{k(t)}'$ are different.
For instance, if this is a query to $g$, then $\hat{g}_0(m_{k(t)})\not=\hat{g}_1(m_{k(t)}')$.
Let us now define $\tilde{m}_t:=m_{k(t)}$.
The point is that if $\hat{g}=\Phi_e[f]$ and $\hat{h}=\Phi_k[f]$, then one can recover the value of $f(n_t)$ from the value of $\hat{g}(\tilde{m}_t)$ or $\hat{h}(\tilde{m}_t)$.
Moreover, even if $g\subseteq\hat{g}$ and $h\subseteq\hat{h}$, if $\Phi_j[g\oplus h](n_t)\downarrow$ then the same computation is performed without accessing outside the domains of $g$ and $h$, so one can recover the value of $f(n_t)$ from the value of ${g}(\tilde{m}_t)$ or $h(\tilde{m}_t)$ in the same way.
This is the case, for example, when $n_t\in{\rm dom}(f)$ and $f\subT g\oplus h$ via $j$.

\medskip
\noindent
{\em Infinite substages:}
Recall that, if a substrategy does not reach (4b) for some value of $f(n_t)$, then we proceed to the next stage.
Therefore, the process to proceed to the next stage is not yet described if all substrategies $t$ reach (4b) for any value of $f(n_t)$.
Also, a new substrategy is activated only when it reaches (4b).
Note that no action is taken at this point.

If each substrategy $t$ reaches (4b), then we can use the information on the finite function $f\upto n_0$ to compute the sequences of queries $m_0,m_1,\dots$ and $m_0',m_1',\dots$, and eventually find a query $\tilde{m}_t$ such that $\hat{g}_0(\tilde{m}_t)\not=\hat{g}_1(\tilde{m}_t)$ or $\hat{h}_0(\tilde{m}_t)\not=\hat{h}_1(\tilde{m}_t)$.
This procedure $t\mapsto\tilde{m}_t$ is computable using only $f\upto n_0$.
Furthermore, for $\hat{g}=\Phi_e[f]$ and $\hat{h}=\Phi_k[f]$, the value of $f(n_t)$ is determined from $\hat{g}(\tilde{m}_t)$ or $\hat{h}(\tilde{m}_t)$.
Let $J_s\subseteq I_s$ be the set of locations $n_t$ such that $\tilde{m}_t$ is a query to the $g$-side and $\hat{g}_0(\tilde{m}_t)\not=\hat{g}_1(\tilde{m}_t)$.
If $J_s$ is infinite, put $I_{s+1}=J_s$, otherwise $I_{s+1}=I_s\setminus J_s$.
We then proceed to stage $s+1$.

\medskip
\noindent
{\em Verification:}
The requirement $P_e$ clearly ensures that $f$ is not computable.
Suppose now that $g,h\subT f$ and $f\subT g\oplus h$.
Then $g\subseteq\Phi_e[f]$, $h\subseteq\Phi_k[f]$ and $f\subseteq\Phi_j[g\oplus h]$ for some $e,k,j\in\om$.
If some substrategy of the $R_{e,k,j}$-strategy does not arrive at (4b), then $f\not\subseteq\Phi_j[g\oplus h]$ is guaranteed as discussed above, so all substrategies will always arrive at (4b).
Let $s=2\langle e,k,j\rangle+1$ and consider $n_0=\min I_s$.
By the construction of $f$, for any $n\in{\rm dom}(f)$, $n\geq n_0$ implies $n\in I_{s+1}$.
For any $n_t\in I_{s+1}$, one can compute $\tilde{m}_t$ using information on the finite function $f\upto n_0$.
As noted above, if $n_t\in{\rm dom}(f)$, then by our assumption that $f\subT g\oplus h$ via $j$, the value of $f(n_t)$ can be computed from $g(\tilde{m}_t)$ or $h(\tilde{m}_t)$.
If $n_t\in J_s$, then the former, and if $n_t\not\in J_s$, then the latter.
Thus, we obtain $f\subT g$ if $I_{s+1}=J_s$ and $f\subT h$ if $I_{s+1}=I_s\setminus J_s$.
\end{proof}

When translated into the context of realizability theory using Fact \ref{fact:realizability-topos}, this result means that there is a realizable subtopos of the effective topos that cannot be decomposed into two smaller realizability subtoposes (with respect to the ordering by geometric inclusions).

\begin{ack}
The authors wish to thank David Belanger for valuable discussions.
The first-named author was supported by JSPS KAKENHI (Grant Numbers 21H03392, 22K03401 and 23H03346).
The second-named author was supported by the Ministry of Education, Singapore, under its Academic Research Fund Tier 2 (MOE-T2EP20222-0018).
\end{ack}

\bibliographystyle{plain}
\bibliography{references}
\end{document}